\documentclass [reqno,12pt,a4paper]{amsart}
\usepackage[a4paper,margin=1in,footskip=0.4in,heightrounded]{geometry}
\usepackage{lastpage}
\usepackage{t1enc}
\usepackage[utf8]{inputenc}
\usepackage{graphicx}
\usepackage{tabularx}
\usepackage{caption,subcaption}
\usepackage{multirow}
\usepackage{tikz,color}
\usepackage{amsmath,mathtools,amsthm}
\usepackage{latexsym,amssymb,mathabx}
\usepackage{epstopdf}
\usetikzlibrary{arrows}
\graphicspath{ {./eps/n4/} }
\tikzset{every loop/.style={}}
\usepackage[pagewise]{lineno} 

\newcommand{\var}{\textnormal{var}}
\newcommand{\supp}{\textnormal{supp}}

\newtheorem{theo}{Theorem}
\newtheorem{prop}[theo]{Proposition}
\newtheorem{lem}[theo]{Lemma}

\newtheorem{claim}[theo]{Claim}
\theoremstyle{definition}

\theoremstyle{remark}
\newtheorem*{rem}{Remark}

\makeatletter
\renewcommand\p@enumi{\arabic{section}.}
\makeatother

\title[Synchronization versus stability in globally coupled maps]{Synchronization versus stability of the invariant distribution for a class of globally coupled maps}
\author[P\'eter B\'alint, Gerhard Keller, Fanni M. S\'elley and Imre P\'eter T\'oth]{P\'eter B\'alint, Gerhard Keller, Fanni M. S\'elley and Imre P\'eter T\'oth}

\address{Fanni M. S\'elley: Alfr\'ed R\'enyi Institute for Mathematics \\
Hungarian Academy of Sciences \\
13-15 Re\'altanoda u. H-1053 Budapest, Hungary
and
MTA-BME Stochastics Research Group \\
Budapest University of Technology and Economics \\
Egry J\'ozsef u. 1, H-1111 Budapest, Hungary
and
Department of Stochastics, Institute of Mathematics \\
Budapest University of Technology and Economics \\
Egry J\'{o}zsef u. 1, H-1111 Budapest, Hungary.}

\address{Gerhard Keller: Department of Mathematics, University of Erlangen-Nuremberg, Cauerstr. 11, D-91058 Erlangen, Germany}

\address{P\'eter B\'alint and Imre P\'eter T\'oth:
MTA-BME Stochastics Research Group \\
Budapest University of Technology and Economics \\
Egry J\'ozsef u. 1, H-1111 Budapest, Hungary
and
Department of Stochastics, Institute of Mathematics \\
Budapest University of Technology and Economics \\
Egry J\'{o}zsef u. 1, H-1111 Budapest, Hungary.}

\date{\today}

\begin{document}

\begin{abstract}
We study a class of globally coupled maps in the continuum limit, where the individual maps are expanding maps of the circle. The circle maps in question are such that the uncoupled system admits a unique absolutely continuous invariant measure (acim), which is furthermore mixing. Interaction arises in the form of diffusive coupling, which involves a function that is discontinuous on the circle. We show that for sufficiently small coupling strength the coupled map system admits a unique  absolutely continuous invariant distribution, which depends on the coupling strength $\varepsilon$. Furthermore, the invariant density exponentially attracts all initial distributions considered in our framework. We also show that the dependence of the invariant density on the coupling strength $\varepsilon$ is Lipschitz continuous in the BV norm.

When the coupling is sufficiently strong, the limit behavior of the system is more complex. We prove that a wide class of initial measures approach a point mass with support moving chaotically on the circle. This can be interpreted as synchronization in a chaotic state.
\end{abstract}

\maketitle

\let\thefootnote\relax
\let\thefootnote\relax

\section{Introduction}

In this paper we investigate a model of globally coupled maps. The precise definition is given in Section~\ref{s:results}; nonetheless, let us summarize here that by a globally coupled map we mean the following setup:
\begin{itemize}
\item The phase space $M$ is a compact metric space, and the state of the system is described by a Borel probability measure $\mu$ on $M$.
\item The dynamics, to be denoted by $F_{\varepsilon,\mu}$, is a composition of two maps $T \circ \Phi_{\varepsilon,\mu}$, where $T:M\to M$ describes the evolution of the individual sites, while $\Phi_{\varepsilon,\mu}:M\to M$ describes the coupling. Here the parameter $\varepsilon\ge 0$ is the coupling strength, and thus $\Phi_{0,\mu}$ is the identity for any $\mu$.
\item If the initial state of the system is given by some probability measure $\mu_0$, then for later times $n\ge 1$ the state of the system is given by $\mu_n$ which is
obtained by pushing forward $\mu_{n-1}$ by the map $F_{\varepsilon,\mu_{n-1}}$.
\end{itemize}

This framework allows a wide range of examples. To narrow down our analysis, let us assume that $M$ is either the interval $[0,1]$ or the circle $\mathbb{T}=\mathbb{R} / \mathbb{Z}$, and that the map $T$ has positive Lyapunov exponent and good ergodic properties with respect to an absolutely continuous invariant measure. Furthermore, we assume that the coupling map is of the form
\begin{equation} \label{couplingG}
\Phi_{\varepsilon,\mu}(x)=x+\varepsilon\cdot G(x,\mu)
 \end{equation}
for some fixed function $G$ such that the dynamics preserves certain classes of measures. In particular, for any $N\ge 1$ the class of measures when $\mu=\frac{1}{N} \sum_{i=1}^N \delta_{x_i}$ for some points $x_1,\dots,x_N\in M$, is preserved, a case that is regarded as a finite system of mean field coupled maps. In turn, the situation when $\mu$ is (Lebesgue-)absolutely continuous can be thought of as the infinite -- more precisely, continuum -- version of the model.

Coupled dynamical systems in general, and globally coupled maps in particular have been extensively studied in the literature. Below we mention the papers that directly motivated our work, more complete lists of references can be found for instance in \cite{koiller2010coupled}, \cite{fernandez2014breaking}, \cite{keller2006uniqueness} and \cite{chazottes2005dynamics}. Our main interest here is to understand how different types of asymptotic phenomena arise in the system depending on the coupling strength $\varepsilon$. In particular, we would like to address the following questions:
\begin{itemize}
\item \textbf{Weak coupling.}
\begin{itemize}
\item Is there some positive $\varepsilon_0$ such that for $\varepsilon\in[0,\varepsilon_0]$ there exists a unique absolutely continuous invariant probability measure $\mu_*=\mu_*(\varepsilon)$? \item Is $\mu_*(\varepsilon)$ and $\mu_*(0)$ close in some sense?
\item Do we have convergence of $\mu_n$ to $\mu_*$, at least for sufficiently regular initial distributions $\mu_0$?
\end{itemize}
\item \textbf{Strong coupling.}
\begin{itemize}
\item Is it true that for larger values of $\varepsilon$ the asymptotic behavior can be strikingly different?
\item In particular, at least for a substantial class of initial distributions $\mu_0$, does $\mu_n$ approach a point mass that is evolved by the chaotic map $T$?
\end{itemize}
\end{itemize}

As by assumption $T$ has good ergodic properties, the weak coupling phenomena mean that for small enough $\varepsilon$ the behavior of the coupled system is analogous to that of the uncoupled system ($\varepsilon=0$). Hence from now on we will refer to this shortly as \textit{stability}. On the other hand, the limit behavior associated to a point mass moving chaotically on $\mathbb{T}$, which is expected to arise for a class of initial distributions when the coupling is sufficiently strong, will be referred to as \textit{synchronization in a chaotic state}. This terminology is borrowed from the literature of coupled map lattices (see eg.~\cite{chazottes2005dynamics}) where synchronization is defined as the analogous phenomena when the diagonal attracts all orbits from an open set of the phase space, which has a direct product structure.

The short summary of our paper is that for the class of models studied here the answer to all of the questions stated above is yes, in the sense formulated in Theorems~\ref{theomain},~\ref{stability} and~\ref{largeeps} below.

Before describing the works that provided a direct motivation for our research, we would like to comment briefly how the above questions can be regarded in the general context of coupled dynamical systems. As already mentioned, this area has an enormous literature, which we definitely do not aim to survey here. In particular, there is a wide range of asymptotic phenomena that may arise depending on the specifics of the coupled dynamical system.
Yet, the two extremes of stability for weak coupling, and synchronization for strong coupling, is a common feature of many of the examples. For instance, in the context of coupled map lattices, the case of small coupling strength can be often treated as the perturbation of the uncoupled system. As a consequence, the sites remain weakly correlated, typically resulting in a unique space-time chaotic phase, which is analogous to the unique absolutely continuous invariant measure in our setting. On the other hand, if the coupling is strong enough, synchronization occurs in the sense that initial conditions are attracted by some constant configurations, in many cases, by the diagonal. The dynamics along the diagonal is then governed by the local map $T$, which has a positive Lyapunov exponent, justifying the terminology of chaotic synchronization. For further discussion of coupled map lattices, we refer to \cite{chazottes2005dynamics}. Although the specifics are quite different, the two extremities of uncorrelated behavior versus synchronization are highly relevant in another popular paradigm for applied dynamics, the Kuramoto model of coupled oscillators (see \cite{dietert2018mathematics} for a recent survey). We emphasize that the asymptotic behavior in coupled map lattices or in the Kuramoto model is much more complex, discussing which is definitely beyond the scope of the present paper. However, it is worth mentioning that our findings are analogous to some key features of these important models.

The question of stability for small $\varepsilon$ in globally coupled maps received some attention in the 1990's when unexpected behavior, often referred to as the violation of the law of large numbers, was observed even for arbitrarily small values of $\varepsilon$, when $T$ was chosen as a quadratic map or a tent map of the interval (\cite{kaneko1990globally}, \cite{kaneko1995remarks}, \cite{ershov1995mean}, \cite{nakagawa1996dominant}). For the description of the violation of the law of large numbers we refer to \cite{keller2000ergodic}, here we only mention that these complicated phenomena definitely rule out the stability for weak coupling scenario as defined above. However, in \cite{keller2000ergodic} it was shown that there is no violation of the law of large numbers, and in particular, there is stability for weak coupling when
\begin{enumerate}
\item[(I)] the map $T$ is a $C^3$ expanding map of the circle $\mathbb{T}=\mathbb{R} / \mathbb{Z}$;
\item[(II)] the coupling $\Phi_{\varepsilon,\mu}$ defined by \eqref{couplingG} is such that
\[
G(x,\mu)=\gamma(x,\bar{\mu}), \quad \bar{\mu}=\int_{\mathbb{T}} F\circ T \text{ d}\mu,
\]
where $\gamma \in C^{2}(\mathbb{T} \times \mathbb{R},\mathbb{R})$ and $F\in C^{2}(\mathbb{T},\mathbb{R})$.
\end{enumerate}

Motivated in part by these results, the authors in \cite{bardet2009stochastically} investigated a class of models where $T$ is the doubling map, which is deformed by a coupling factor $G(x,\mu)=\gamma(x,\bar\mu)$ to a nonlinear piecewise fractional linear map. For this class of examples, which do not literally fit the structure described by \eqref{couplingG}, \cite{bardet2009stochastically} proved that while there is stability for weak coupling, a phase transition analogous to that of the Curie-Weiss model takes place for stronger coupling strength.

A parallel line of investigation was initiated when a class of models was introduced in \cite{koiller2010coupled}, and later studied from the ergodic theory point of view in \cite{fernandez2014breaking} and \cite{selley2016mean}. In these papers $T$ is the \textit{doubling map} of $\mathbb{T}$, while $\Phi_{\varepsilon,\mu}$ represents
\textit{a diffusive coupling on the circle} among the particles. As such, $\Phi_{\varepsilon,\mu}$ is given by the formula (\ref{couplingG}), where
\[
G(x,\mu)=\int_{\mathbb{T}}g(x-y)\text{ d}\mu(y)
\]
for the function $g$ depicted on Figure~\ref{g}, which is \textit{discontinuous on the circle}. As discussed in \cite{fernandez2014breaking} and \cite{selley2016mean}, the discontinuity of $g$ has some important consequences for the behavior of the finite system. In particular it is shown, mathematically for $N=3$ (\cite{fernandez2014breaking} and \cite{selley2016mean}) and $N=4$ (\cite{selley2016}), and numerically for higher values of $N$ (\cite{fernandez2014breaking}),  that a loss of ergodicity takes place for finite system size when the coupling strength is increased beyond a critical value. From a different perspective, in \cite{selley2016mean}  we studied the continuum version of this model as well. In that context we showed that while there is stability for weak coupling (\cite[Theorem 4]{selley2016mean}), yet, when $\varepsilon$ is sufficiently large, for a substantial class of initial distributions, $\mu_n$ approaches a point mass with support moving chaotically on $\mathbb{T}$ (as formulated in \cite[Theorem 5]{selley2016mean}).

In this paper we go beyond the doubling map discussed in \cite{selley2016mean}. The direct motivation
for this is that in that model, stability was, to some extent, prebuilt into the system: the very same (i.e. Lebesgue) measure
was invariant for every coupling strength.
Without such a symmetry, it was not at all clear whether the invariant measure could be stable under
perturbation by the discontinuous $g$. In other words, we are looking for the regularity class of systems where
the phase transition (or possibly a sequence of phase transitions) between stable behavior and synchronization occurs at positive coupling strength -- of which \cite{selley2016mean} only gave an example. Now we are able to demonstrate that the symmetry of the doubling map is not needed, and much less regularity is enough.
In particular, we consider the model of
\cite{selley2016mean}, yet, instead of the doubling map, $T$ is now basically an arbitrary $C^2$ expanding map of the circle (see Section~\ref{s:results} below for the precise formulation). Our main results are stability for weak coupling as formulated in Theorem~\ref{theomain} and ~\ref{stability}, and synchronization in a chaotic state for strong coupling as formulated in Theorem~\ref{largeeps}. Theorem~\ref{theomain} can be regarded as generalization of \cite[Theorem 4]{selley2016mean}, while Theorem~\ref{largeeps} is a generalization of \cite[Theorem 5]{selley2016mean} to this context.
In \cite{selley2016mean} the $\varepsilon$-dependence of the (unique absolutely continuous) invariant measure was not a question, since the measure was always Lebesgue by construction. In the present model this is not the case. Instead,
we prove in Theorem~\ref{stability} that the unique invariant density depends Lipschitz continuously on the coupling strength.

It is important to emphasize that our methods here are quite different from those of \cite{selley2016mean}. There, our arguments were somewhat ad hoc, based on computations that exploited some specific features, in particular the linearity of the doubling map. Here, a more general approach is needed; we apply spectral tools developed in the literature. In that respect, the proof of our Theorem~\ref{theomain} is strongly inspired by \cite{keller2000ergodic}. Yet, we would like to point out an important difference.
The assumptions of \cite{keller2000ergodic}, in particular \cite[Theorem 4]{keller2000ergodic} are summarized in (I) and (II) above. For the model discussed here, the map $T$ is not much different, though it is slightly less regular than assumed in (I). On the other hand, the coupling $\Phi_{\varepsilon,\mu}$ differs considerably from the one defined in assumption (II). It does not depend only on an integral average of $\mu$, but on a more complicated expression involving the function $g$, which \textit{is discontinuous on} $\mathbb{T}$. As mentioned above, the discontinuity of $g$ has some important consequences for the finite system size. In the continuum version of the model, it turns out that we have stability for weak coupling (Theorem~\ref{theomain}) as in the setting of \cite[Theorem 4]{keller2000ergodic}. Yet, $g$ causes several subtle technical challenges in the proof of Theorem~\ref{theomain}, since the discontinuities imply that not an integral expression, but a certain value of the density $f$ will play a role in $F_{\varepsilon,\mu}'$ -- and the same can be said for $f'$ and $F_{\varepsilon,\mu}''$. A related comment we would like to make concerns Theorem~\ref{stability} which proves that the invariant density depends on $\varepsilon$ Lipschitz continuously. We think that this result is remarkable as in the presence of singularities typically only a weaker, log-Lipschitz continuous dependence can be expected (\cite{baladi2007susceptibility}, \cite{bonetto2000properties}, \cite{keller2008continuity}).

The remainder of the paper is organized as follows: in Section~\ref{s:results} we introduce our model, and state our three theorems concerning stability and synchronization. In Section~\ref{s:theo1proof} we prove our first theorem on stability: we show that there exists an invariant absolutely continuous distribution, which is unique in our setting. We then show that densities close enough to the invariant density converge to it with exponential speed. In Section~\ref{s:theo2proof} we prove our second theorem concerning the Lipschitz continuity of the invariant density in $\varepsilon$. In Section~\ref{s:theo3proof} we prove our theorem on synchronization, namely we show that for large enough $\varepsilon$, sufficiently well-concentrated initial distributions tend to a point mass moving on $\mathbb{T}$ according to $T$.

In the proof of our statements, especially those about stability, the choice of a suitable space of densities plays an important role. This choice is discussed in Section~\ref{s:remarks}, along with some further open problems.

\section{The model and main results}
\label{s:results}

Let $\mathbb{T}=\mathbb{R} / \mathbb{Z}$, and denote the Lebesgue measure on $\mathbb{T}$ by $\lambda$. Consider the Lebesgue-absolutely continuous probability measure $\text{d}\mu=f\text{d}\lambda$ on $\mathbb{T}$ and a self-map of $\mathbb{T}$ denoted by $T$. We make the following initial assumptions on them:
\begin{enumerate}
\item[(F)] $f \in C^1(\mathbb{T},\mathbb{R}_{\geq 0})$, $f'$ is furthermore Lipschitz continuous and $\int f \text{ d}\lambda=1$,
\item[(T)] $T \in C^2(\mathbb{T}, \mathbb{T})$, $T''$ is Lipschitz continuous and $T$ is strictly expanding: that is, $\min|T'| = \omega > 1$. We further suppose that $T$
is $N$-fold covering and
\[
N < \omega^2.
\]
\end{enumerate}
Define $\Phi_{\mu}: \mathbb{T} \to \mathbb{T}$ as
\[
\Phi_{\mu}(x)=x+\varepsilon \int_{\mathbb{T}} g(y-x)\text{d}\mu(y) \qquad x \in \mathbb{T},
\]
where $0 \leq \varepsilon < 1$ and $g: \mathbb{T} \to \mathbb{R}$ is defined as
\begin{equation} \label{gg}
g(u)=
\begin{cases}
u & \text{if } u \in \left(-\frac{1}{2},\frac{1}{2} \right), \\
0 & \text{if } u =\pm \frac{1}{2}.
\end{cases}
\end{equation}
The graph of the natural lift of this function to $\mathbb{R}$ is depicted
in Figure \ref{g}.

\begin{figure}[h!]
 \centering
 \begin{tikzpicture}[scale=2]
       \draw[->] (-2.5,0) -- (2.5,0) node[above] {$u$};
       \draw[->] (0,-1) -- (0,1) node[right] {\hspace{0.1cm}$g(u)$};
       \draw[very thick,red] (-0.7,-0.7) -- (0.7,0.7);
       \draw[very thick,red] (0.7,-0.7) -- (2.1,0.7);
       \draw[very thick,red] (-0.7,0.7) -- (-2.1,-0.7);
       \draw[red, fill] (0.7,0) circle (0.03cm);
       \draw[red, fill] (-0.7,0) circle (0.03cm);
       \draw[red, fill] (2.1,0) circle (0.03cm);
       \draw[red, fill] (-2.1,0) circle (0.03cm);
       \foreach \x/\xtext in {0.7/\frac{1}{2},1.4/1, 2.1/\frac{3}{2},-0.7/-\frac{1}{2},-1.4/-1, -2.1/-\frac{3}{2}}
           \draw[shift={(\x,0)}] (0pt,2pt) -- (0pt,-2pt) node[below] {$\xtext$};
       \foreach \y/\ytext in {0.7/\frac{1}{2},-0.7/-\frac{1}{2}}
             \draw[shift={(0,\y)}] (2pt,0pt) -- (-2pt,0pt) node[left] {$\ytext$};
     \end{tikzpicture}
     \caption{The function $g$.} \label{g}
     \end{figure}
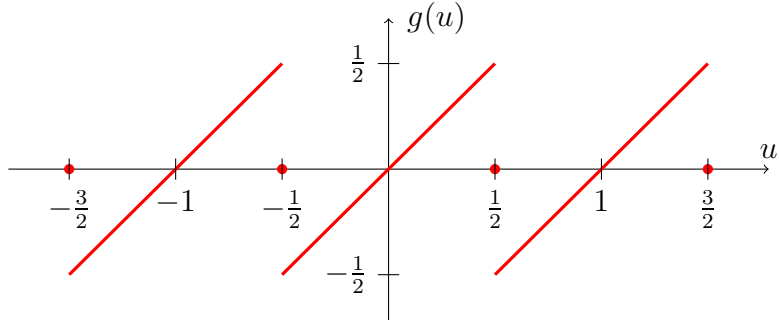
Define $F_{\mu}: \mathbb{T} \to \mathbb{T}$ as
\[
F_{\mu}= T \circ \Phi_{\mu}.
\]
This can be regarded as a coupled map dynamics in the following way:  $\Phi_{\mu}$ accounts for the interaction between the sites (distributed according to the measure $\mu$) via the interaction function $g$. The map $T$ is the individual site dynamics. We call the parameter $\varepsilon$ coupling strength.

Let $\mu_0$ be the initial distribution and define
\begin{equation} \label{pforward}
\mu_{n+1}=(F_{\mu_n})_{\ast}\mu_n, \quad n=0,1,\dots
\end{equation}

Our assumptions guarantee that if $\mu_0 \ll \lambda$ with density $f_0$ of property (F) then $\mu_n \ll \lambda$ with density of property (F) for all $n \in \mathbb{N}$ (this will be proved later). We are going to use the notation $\text{d}\mu_n=f_n\text{d}\lambda$ for the densities. Also, we are going to index $F_{\cdot}$ and $\Phi_{\cdot}$ with the density instead of the measure. Now $f_{n+1}$ can be calculated with the help of the transfer operator $P_{F_{f_n}}$ in the following way:
\[
f_{n+1}(y)=P_{F_{f_n}}f_n(y)=\sum_{x: \thinspace  F_{f_n}(x)=y}\frac{f_n(x)}{|F_{f_n}'(x)|}, \qquad y \in \mathbb{T}.
\]
To simplify notation, we are going to write
\[
P_{F_f}f=\mathcal{F}_{\varepsilon} (f),
\]
so $f_{n+1}=\mathcal{F}_{\varepsilon} (f_n)$.

Our main goal is to show that depending on $\varepsilon$, there exist two different limit behaviors of the sequence $(\mu_n)$.

The first type of limit behavior occurs when the coupling is sufficiently weak in terms of the regularity of the initial distribution. In this case, we claim that there exists an invariant distribution with density of property (F), and initial distributions sufficiently close to it converge to this distribution exponentially. More precisely, let
\[
\mathcal{C}_{R,S}^c=\{ f \text{ is of property (F)}, \text{var}(f) \leq R, |f'| \leq S, \text{Lip}(f') \leq c \}
\]
where we denoted the total variation of the function $f$ by var$(f)$, its Lipschitz constant by Lip($f$) and $R,S,c > 0$.

The choice of the set of densities $\mathcal{C}_{R,S}^c$ plays a central role in the arguments -- see the discussion in Section~\ref{s:discussion}. It is a subset of the space of functions of bounded variation, so we can endow it with the usual bounded variation norm:
\[
\|f\|_{BV}=\|f\|_1+\text{var}(f),
\]
where
\[
\|f\|_1=\int |f| \text{ d}\lambda \quad \text{and} \quad \text{var}(f)=\int |f'| \text{ d}\lambda.
\]
Note that the total variation $\text{var}(f)$ can be calculated indeed with this simple formula, since $f$ is continuously differentiable.

\begin{theo} \label{theomain}
There exist $R^*,S^*$ and $c^*>0$  such that for all
$R>R^*$, $S>S^*$ and $c>c^*$ there exists an $\varepsilon^*=\varepsilon^*(R,S,c)>0$, for which the following holds:
For all $0 \leq \varepsilon < \varepsilon^*$, there exists a density $f_*^{\varepsilon} \in \mathcal{C}_{R,S}^c$ for which $\mathcal{F}_{\varepsilon}f_*^{\varepsilon}=f_*^{\varepsilon}$. Furthermore,
\[
\lim_{n \to \infty}\mathcal{F}_{\varepsilon}^n(f_0)=f_{*}^{\varepsilon} \quad \text{ exponentially for all } f_0  \in \mathcal{C}_{R,S}^c
\]
in the sense that there exist $C > 0$ and $\gamma\in(0,1)$ such that
\[
\|\mathcal{F}_{\varepsilon}^n(f_0)-f_*^{\varepsilon}\|_{BV} \leq C \gamma^n \|f_0-f_*^{\varepsilon}\|_{BV} \quad \text{ for all } n \in \mathbb{N}
\text{ and } 0\leq \varepsilon<\varepsilon^*.
\]
\end{theo}

In this case we can also show that the fixed density of $\mathcal{F}_{\varepsilon}$ is Lipschitz continuous in the variable $\varepsilon$.

\begin{theo}  \label{stability}
Let $R,S,c$ and $\varepsilon^{*}$ be chosen as in Theorem~\ref{theomain}. Then there exists a $K(R,S,c)=K > 0$ such that for any $0 \leq \varepsilon, \varepsilon' < \varepsilon^{*}$
\[
\|f_*^{\varepsilon}-f_*^{\varepsilon'}\|_{BV} \leq K|\varepsilon-\varepsilon'|
\]
holds for the densities $f_*^{\varepsilon}, f_*^{\varepsilon'} \in \mathcal{C}_{R,S}^c$ for which
$\mathcal{F}_{\varepsilon}f_*^{\varepsilon}=f_*^{\varepsilon}$ and $\mathcal{F}_{\varepsilon'}f_*^{\varepsilon'}=f_*^{\varepsilon'}$.
\end{theo}

However, when the coupling is strong, we expect to see synchronization in some sense. To be able to prove such behavior we need an initial distribution which is `sufficiently well concentrated' -- by this we mean the following:
\begin{enumerate}
\item[(F')] $f$ is of property (F), furthermore, there exists an interval $I \subset \mathbb{T}$, $|I| \geq \frac{1}{2}$ such that $\supp(f) \cap I = \emptyset$.
\end{enumerate}
In this case we can define $\supp^* (f)$ as the smallest closed interval on $\mathbb{T}$ containing the support of $f$.

Before stating our theorem, we recall the definition of the Wasserstein metric. Let $(S,d)$ be a metric space and let $\mathcal{P}_1(S)$ denote all measures $P$ for which $\int d(x,z)\text{ d}P(x) < \infty$ for every $z \in S$. Let $M(P,Q)$ be the set of measures on $S \times S$ with marginals $P$ and $Q$. Then the Wasserstein distance of $P$ and $Q$ is
\[
W_1(P,Q)=\inf \left\{\int d(x,y) \text{ d}\mu(x,y), \text{ } \mu \in M(P,Q) \right\}.
\]
By the Kantorovich-Rubinstein theorem {\cite[Theorem 11.8.2]{dudley2002real}} it holds that
\[
W_1(P,Q)=\sup \left\{\left|\int f \text{ d}(P-Q)\right|, \text{ } \text{Lip}(f) \leq 1 \right\}.
\]
\begin{theo} \label{largeeps}
Suppose $1-\frac{1}{\max |T'|} < \varepsilon < 1$. Then
\[
|\supp^* (f_n)| \underset{n \to \infty}{\to} 0 \quad \text{exponentially}
\]
in the sense that
\[
|\supp^* (f_n)| \leq [\max |T'|(1-\varepsilon)]^n |\supp^* (f_0)| \quad \text{for all} \quad n \in \mathbb{N}.
\]
Furthermore, there exists an $x^* \in \supp^*(f_0)$ such that
\[
W_1(\mu_n,\delta_{T^n(x^*)}) \underset{n \to \infty}{\to} 0 \quad \text{exponentially}
\]
in the sense that
\[
W_1(\mu_n,\delta_{T^n(x^*)}) \leq [\max |T'|(1-\varepsilon)]^n W_1(\mu_0,\delta_{x^*}) \quad \text{for all} \quad n \in \mathbb{N}.
\]
\end{theo}

So we claim that when the coupling is sufficiently strong, the support of a well-concentrated initial density eventually shrinks to a single point, hence complete synchronization is achieved.

We have chosen the Wasserstein metric to state our theorem because the convergence in this metric is equivalent to weak convergence of measures. This is about the best that can be expected, since for example a similar statement for the total variation distance cannot hold - a sequence of absolutely continuous measures cannot converge to a point measure in total variation distance.

In the subsequent three chapters we are going to give the proofs of Theorems~ \ref{theomain},~\ref{stability} and ~\ref{largeeps}.

\section{Proof of theorem \ref{theomain}}
\label{s:theo1proof}

\subsection{Existence of an invariant density} In this section we prove the following proposition:

\begin{prop} \label{propfix}
There exist $\varepsilon^*_0,R^*,S^*,c^* > 0$ such that if $0 \leq \varepsilon < \varepsilon^*_0$, the operator $\mathcal{F}_{\varepsilon}$ has a fixed point $f_*^{\varepsilon}$ in $\mathcal{C}_{R^*,S^*}^{c^*}$.
\end{prop}
\begin{proof}
The structure of the proof is as follows: we first show that $\mathcal{C}_{R,S}^c$ is invariant under the action of $\mathcal{F}_{\varepsilon}$ if $R,S$ and $c$ are chosen large enough. Then we prove that $\mathcal{F}_{\varepsilon}$ restricted to $\mathcal{C}_{R,S}^c$ is continuous in the $L^1$ norm. We then argue that the set $\mathcal{C}_{R,S}^c$ is a compact, convex metric space for any values of $R,S$ and $c$. Finally, we conclude by Schauder's fixed point theorem that $\mathcal{F}_{\varepsilon}$ has a fixed point in $\mathcal{C}_{R,S}^c$.

\begin{lem} \label{lem:1}
There exists $R^*>0$ such that for all $R\geq R^*$ there are $\varepsilon_0^*=\varepsilon_0^*(R)>0$ and $S^*=S^*(R)>0$ with the following properties:
For each $S\geq S^*$ there is $c^*=c^*(R,S)$
 such that
$\mathcal{F}_{\varepsilon}(\mathcal{C}_{R,S}^c) \subseteq \mathcal{C}_{R,S}^c$ for $c \geq c^*$ and $0 \leq \varepsilon < \varepsilon^*_0$.
\end{lem}
\begin{proof}
The proof consist of the following steps: we first prove that if $f$ is of property (F), then $\mathcal{F}_{\varepsilon}(f)$ is also of property (F) -- as indicated in the introduction. Then let $f \in \mathcal{C}_{R,S}^c$ for some $R,S,c > 0$. We prove that we can choose $R$ large enough such that var$(\mathcal{F}_{\varepsilon}(f)) \leq R$. Then we prove that we can choose $S$ large enough such that $|(\mathcal{F}_{\varepsilon}(f))'| \leq S$ and $c$ large enough such that $\text{Lip}(\mathcal{F}_{\varepsilon}(f)') \leq c$ (provided that $\varepsilon$ is small enough).

\textbf{$\mathcal{F}_{\varepsilon}(f)$ is also of property (F).} First notice that
\[
\Phi'_{f}(x)=1+\varepsilon\left(f\left(x+\frac{1}{2}\right)-1 \right)
\quad \text{and} \quad \Phi''_{f}(x)=\varepsilon f'\left(x+\frac{1}{2}\right).
\]
This implies that $\Phi_{f}$ is monotone increasing:
\begin{equation} \label{phibound}
\Phi_{f}' \geq 1+\varepsilon (\inf_{\mathbb{T}} f-1) \geq 1-\varepsilon > 0 \qquad \text{if } 0 \leq \varepsilon < 1,
\end{equation}
and onto, since denoting the lift of $\Phi_f$ to $\mathbb{R}$ by $\Phi_{f}^{\mathbb{R}}$ we can see that
\[
\Phi_{f}^{\mathbb{R}}(0)=\varepsilon \int_{\mathbb{T}} g(y)f(y)\text{d}y
\quad \text{and} \quad
\Phi_{f}^{\mathbb{R}}(1)=1+\varepsilon \int_{\mathbb{T}} g(y-1)f(y)\text{d}y,
\]
implying that $\Phi_{f}^{\mathbb{R}}(1)=\Phi_{f}^{\mathbb{R}}(0)+1$. Bearing in mind the differentiability properties of $f$, we can observe that $\Phi_f$ is a $C^2$ diffeomorphism of $\mathbb{T}$. We are going to use the change of variables formula with respect to this diffeomorphism repeatedly in the calculations to come.

A further note on $\Phi_f$: we are going to denote the transfer operator associated to it by $P_{\Phi_f}$. Recall that this depends on $\varepsilon$. Denoting the transfer operator associated to $T$ by $P_T$, we have $P_{F_f}=P_TP_{\Phi_f}$ since $F_f=T \circ \Phi_f$.

Now we can see that $F_f$ is $N$-fold covering of $\mathbb{T}$. Let us denote the inverse branches by $F_f^{-1,k}$, $k=1,\dots,N$. Then
\[
\mathcal{F}_{\varepsilon}(f)=P_{F_{f}}f=\sum_{k=1}^{N}\frac{f \circ F_{f}^{-1,k}}{|F_{f}' \circ F_{f}^{-1,k}|},
\]
and we see that $\mathcal{F}_{\varepsilon}(f)$ is also $C^1$. It is also easy to see that the derivative
\begin{align} \label{der}
\mathcal{F}_{\varepsilon}(f)' = \sum_{k=1}^{N} \frac{f'}{(F_{f}')^2} \circ F_{f}^{-1,k} +\text{sign}(T')\,\cdot\sum_{k=1}^{N} \frac{f\cdot F_{f}''}{(F_{f}')^3} \circ F_{f}^{-1,k}.
\end{align}
is also Lipschitz continuous.

\textbf{Choice of $\mathbf{R}^*$.} Fix $0< b < \omega-1$. Then choose $R_0$ such that $\max|T''|/|T'| < b \cdot R_0$. First we note that
\begin{align*}
\var(P_{T}f) & \leq \int \left|\frac{f'}{T'}\right|\text{ d}\lambda +\int \left|\frac{f\cdot T''}{(T')^2} \right|\text{ d}\lambda \\
& \leq \max \left| \frac{1}{T'} \right |\var(f)+ \|f\|_1 \max \left( \frac{T''}{(T')^2} \right) \\
& \leq \max \frac{1}{|T'|}(\var(f)+bR_0).
\end{align*}

Let $\eta = \max \frac{1}{|T'|}(1+b)$ (note that $\eta < 1$ by our choice of $b$). Choose $\rho > 0$ such that $\delta=(1-\rho)^2-\eta > 0$. Let $R^*=\eta\cdot\max \left\{R_0,\frac{\rho}{\delta} \right\}$. Now given $R>R^*$, choose $\varepsilon_0^*\leqslant\frac{\rho}{R}$.
Later in the proof we will require further smallness properties of $\varepsilon_0^*$.

Fix $R\geqslant R^*$, $\varepsilon\in[0,\varepsilon_0^*)$, and
let $\var(f) \leq R$. As $\Phi_f'\geq 1-\varepsilon\var(f)\geq 1-\rho$ and $\int|\Phi_f''|=\varepsilon\var(f)\leq\rho$, we have
\begin{align*}
\text{var}(P_{\Phi_{f}}f)
& \leq
\var(f) \left( \max \left| \frac{1}{\Phi_{f}'} \right | + \var \left( \frac{1}{\Phi_{f}'} \right) \right)+\|f\|_1 \var \left( \frac{1}{\Phi_{f}'} \right) \\
& \leq
\var(f) \cdot\left( \frac{1}{1-\rho}+ \frac{\rho}{(1-\rho)^2} \right)+\|f\|_1 \frac{\rho }{(1-\rho)^2} \\
&=
\var(f) \frac{1}{(1-\rho)^2}+\frac{\rho}{(1-\rho)^2} \leq
\frac{1}{(1-\rho)^2}(R+\delta\eta^{-1} R^*)
\leqslant \frac{1}{\eta(1-\rho)^2}(\eta R+\delta R)\\
&=
\frac{R}{\eta}
\end{align*}
Finally, if $h=P_{\Phi_{f}}f$ then
\[
\var(P_Th) \leq \max \frac{1}{|T'|}(\eta^{-1}R+bR_0)
\leq \max \frac{1}{\eta\,|T'|}(R+bR^*)
\leq R.
\]
\textbf{Choice of $\mathbf{S}^*$.}
We are going to estimate $\mathcal{F}_{\varepsilon}(f)'$ using \eqref{der}. Since
\begin{align*}
F_{f}' \circ F_{f}^{-1,k}&=(T' \circ T^{-1,k}) \cdot (\Phi_{f}' \circ F_{f}^{-1,k})\\
F_{f}'' \circ F_{f}^{-1,k}&=(T'' \circ T^{-1,k}) \cdot (\Phi_f' \circ F_{f}^{-1,k})^2+(T' \circ T^{-1,k}) \cdot (\Phi_{f}'' \circ F_{f}^{-1,k}), \qquad k=1,\dots,N
\end{align*}
we get the following expression:
\begin{align*}
|\mathcal{F}_{\varepsilon}(f)'| &\leq \sum_{k=1}^{N} \left| \frac{f' \circ F_{f}^{-1,k}}{[(T' \circ T^{-1,k}) \cdot (\Phi_{f}' \circ F_{f}^{-1,k})]^2} \right| \\
&+\sum_{k=1}^{N} \left|\frac{(f \circ F_{f}^{-1,k}) \cdot [(T'' \circ T^{-1,k}) \cdot (\Phi_f' \circ F_{f}^{-1,k})^2+(T' \circ T^{-1,k}) \cdot (\Phi_{f}'' \circ F_{f}^{-1,k})]}{[(T' \circ T^{-1,k}) \cdot (\Phi_{f}' \circ F_{f}^{-1,k})]^3} \right|
\end{align*}
Remember that $\omega = \min|T'|$ and let $D=\max|T''|$. As $|f'|\leq S$
and $\varepsilon\in[0,\varepsilon_0^*)$, this implies
\begin{align*}
|\mathcal{F}_{\varepsilon}(f)'| &\leq \frac{N}{\omega^2(1-\varepsilon)^2}S+N(1+R)\left(\frac{D}{\omega^3(1-\varepsilon)}+\frac{\varepsilon S}{(1-\varepsilon)^3\omega^2} \right) \\
&\leq
\left(\frac{N}{\omega^2(1-\varepsilon_0^*)^2}+\frac{\varepsilon_0^* N(1+R)}{(1-\varepsilon_0^*)^3\omega^2}\right)S+\frac{DN(1+R)}{\omega^3(1-\varepsilon_0^*)} \\
&=: q_0(\varepsilon_0^*,R)\cdot S+K_0(\varepsilon_0^*,R).
\end{align*}
As $N<\omega^2$ by assumption (T), one can choose $\varepsilon_0^*=\varepsilon_0^*(R)$ so small that $q_0(\varepsilon_0^*,R)<1$.
Let
\[
S^*=S^*(R):=\frac{K_0(\varepsilon_0^*(R),R)}{1-q_0(\varepsilon_0^*(R),R)},
\]
and suppose that $S\geq S^*$. Then $|\mathcal{F}_{\varepsilon}(f)'|\leq S$.

\textbf{Choice of $\mathbf{c}^*$.} We are going to estimate $\text{Lip}(\mathcal{F}_{\varepsilon}(f)')$ with the help of \eqref{der}.
\begin{align*}
\text{Lip}(\mathcal{F}_{\varepsilon}(f)') & \leq \sum_{k=1}^{N} \text{Lip} \left( \frac{f'}{(F_{f}')^2} \circ F_{f}^{-1,k}\right) +\sum_{k=1}^{N} \text{Lip}\left(\frac{f\cdot F_{f}''}{(F_{f}')^3} \circ F_{f}^{-1,k} \right), \\
& \leq \frac{N}{\omega(1-\varepsilon_0^*)} \cdot \left( \text{Lip} \left( \frac{f'}{(F_{f}')^2} \right)+ \text{Lip}\left(\frac{f\cdot F_{f}''}{(F_{f}')^3} \right) \right),
\end{align*}
since $\max|(F_{f}^{-1})'| \leq \frac{1}{\omega(1-\varepsilon)} \leq \frac{1}{\omega(1-\varepsilon_0^*)}$.

Simple calculations yield that
\begin{align*}
\text{Lip} \left( \frac{f'}{(F_{f}')^2} \right) & \leq \frac{1}{\omega^2(1-\varepsilon_0^*)^2}c+K'(\varepsilon_0^*,R,S),
\end{align*}
and
\begin{align*}
\text{Lip}\left(\frac{f\cdot F_{f}''}{(F_{f}')^3} \right) \leq \frac{(1+R)\max|T'|\varepsilon_0^*}{\omega^3(1-\varepsilon_0^*)^3}c+K''(\varepsilon_0^*,R,S).
\end{align*}
In conclusion,
\begin{align*}
\text{Lip}(\mathcal{F}_{\varepsilon}(f)')
&\leq
N \left(\frac{1}{\omega^3(1-\varepsilon_0^*)^3}+\frac{(1+R)\max|T'|\varepsilon_0^*}{\omega^4(1-\varepsilon_0^*)^4} \right)c+\frac{N}{\omega(1-\varepsilon_0^*)}(K'+K'') \\
&:=
q_1(\varepsilon_0^*,R)\cdot c+K_1(\varepsilon_0^*,R,S).
\end{align*}
As in the previous step,
one can choose $\varepsilon_0^*=\varepsilon_0^*(R)$ so small that $q_1(\varepsilon_0^*,R)<1$.
Let
\[
c^*=c^*(R,S):=\frac{K_1(\varepsilon_0^*(R),R,S)}{1-q_1(\varepsilon_0^*(R),R)},
\]
and suppose that $c\geq c^*$. Then $\text{Lip}(\mathcal{F}_\varepsilon(f)')\leq c$.
\end{proof}

\begin{lem} \label{lemma_cont}
$\mathcal{F}_{\varepsilon}|_{\mathcal{C}_{R,S}^c}$ is continuous in the $L^1$-norm.
\end{lem}
\begin{proof}
We remind the reader that $\mathcal{F}_{\varepsilon}(f)=P_{F_{f}}f=P_TP_{\Phi_{f}}f$. Continuity of $P_T$ is standard:
\[
\| P_Tf_1-P_Tf_2\|_1 = \|P_T(f_1-f_2)\|_1 \leq \|f_1-f_2\|_1.
\]
So it suffices to see the continuity of $\tilde{\mathcal{F}}_{\varepsilon}(f)=P_{\Phi_{f}}f$. Let $f_1,f_2 \in \mathcal{C}_{R,S}^c$, and for the sake of brevity we are going to write $\Phi_{f_1}=\Phi_1$ and $\Phi_{f_2}=\Phi_2$. Then
\begin{align*}
\|P_{\Phi_1}f_1-P_{\Phi_2}f_2 \|_1 &\leq \|P_{\Phi_1}(f_1-f_2) \|_1 + \|(P_{\Phi_1}-P_{\Phi_2})f_2 \|_1 \\
&\leq \|f_1-f_2 \|_1 + \|(P_{\Phi_1}-P_{\Phi_2})f_2 \|_1
\end{align*}

\begin{claim} \label{l1}
Let $f_1,f_2$ be of property (F) and $\varphi$ be of bounded variation on $\mathbb{T}$. Denote $\Phi_{f_1}=\Phi_1$ and $\Phi_{f_2}=\Phi_2$. Then there exists a $K > 0$ such that
\[
\|(P_{\Phi_1}-P_{\Phi_2})\varphi \|_1 \leq \varepsilon K \|\varphi\|_{BV} \|f_1-f_2\|_1.
\]
\end{claim}
This Claim implies that
\begin{align*}
\|\tilde{\mathcal{F}}_{\varepsilon}(f_1)-\tilde{\mathcal{F}}_{\varepsilon}(f_2)\|_1 &\leq \left(1+\varepsilon \cdot K\|f_2\|_{BV}
\right)\|f_1-f_2\|_1 \\
& \leq \left(1+\varepsilon \cdot \text{const}(R) \right)\|f_1-f_2\|_1,
\end{align*}
hence the Lemma is proved once we have this claim.

\emph{The proof of Claim~\ref{l1}.}
\begin{align*}
\|(P_{\Phi_1}-P_{\Phi_2})\varphi \|_1&=\int \left|\frac{\varphi}{\Phi'_1} \circ \Phi^{-1}_1-\frac{\varphi}{\Phi'_2} \circ \Phi^{-1}_2 \right|\text{ d}\lambda \\ &\leq \underbrace{\int \left|\frac{\varphi}{\Phi'_1} \circ \Phi^{-1}_1-\frac{\varphi}{\Phi'_2} \circ \Phi^{-1}_1 \right|\text{ d}\lambda}_{(A)}+\underbrace{\int \left|\frac{\varphi}{\Phi'_2} \circ \Phi^{-1}_2-\frac{\varphi}{\Phi'_2} \circ \Phi^{-1}_1 \right|\text{ d}\lambda}_{(B)}
\end{align*}

We first deal with the term $(A)$.

\begin{align*}
\left|\int \left(\frac{\varphi}{\Phi_1'} -\frac{\varphi}{\Phi_2'} \right) \circ \Phi_1^{-1}\right|\text{ d}\lambda&=\int |\varphi| \cdot \left|1 -\frac{\Phi_1'}{\Phi_2'} \right|\text{ d}\lambda \leq \|\varphi\|_{\infty} \int  \left|\frac{\Phi_2'-\Phi_1'}{\Phi_2'} \right|\text{ d}\lambda \\
&= \|\varphi\|_{BV} \int_0^1\left|\frac{\varepsilon(f_2(x+1/2)-f_1(x+1/2))}{1+\varepsilon(f_2(x+1/2)-1)} \right| \text{ d}x \\
&\leq \frac{\varepsilon }{1-\varepsilon}\|\varphi\|_{BV} \|f_1-f_2\|_1.
\end{align*}

Now we give an appropriate bound on $(B)$.

\begin{align*}
\int \left|\frac{\varphi}{\Phi_2'} \circ \Phi_2^{-1}-\frac{\varphi}{\Phi_2'} \circ \Phi_1^{-1} \right|\text{ d}\lambda&=\int|\varphi \circ \Phi_1^{-1} \circ \Phi_2-\varphi|\text{ d}\lambda=\int \psi\cdot (\varphi \circ \Phi_1^{-1} \circ \Phi_2-\varphi)\text{ d}\lambda \\
&=\int (P_{\Phi_1^{-1}}P_{\Phi_2}\psi-\psi)\varphi\text{ d}\lambda,
\end{align*}
where $\psi=\text{sign}(\varphi \circ \Phi_1^{-1} \circ \Phi_2-\varphi)$. Now we are going to apply {\cite[Lemma 11]{keller1982stochastic}}, which states that if $\ell \in BV$ and $h \in L^1$, then
\[
\left|\int \ell \cdot h \text{ d}\lambda \right| \leq \text{var}(\ell)\left \| \int (h) \right\|_{\infty}+\left|\int h \text{ d}\lambda \right| \cdot \|\ell\|_{\infty} \leq 2 \|\ell\|_{BV}\cdot\left \| \int (h) \right\|_{\infty},
\]
where $\int (h)=\int_{\{x \leq z\}}h(x)\text{ d}x$. Choosing $\ell=\varphi$ and $h=P_{\Phi_1^{-1}}P_{\Phi_2}\psi-\psi$ we get
\[
\int (P_{\Phi_1^{-1}}P_{\Phi_2}\psi-\psi)\varphi\text{ d}\lambda \leq 2 \|\varphi\|_{BV}\sup_{0 \leq z \leq 1}\left| \int (P_{\Phi_1^{-1}}P_{\Phi_2}\psi-\psi) \mathbf{1}_{[0,z]}  \right|\text{ d}\lambda.
\]
So
\begin{align*}
\int (P_{\Phi_1^{-1}}P_{\Phi_2}\psi-\psi)\varphi\text{ d}\lambda &\leq 2\|\varphi\|_{BV} \sup_{0 \leq z \leq 1}\left| \int \psi \mathbf{1}_{[0,z]} \circ \Phi_1^{-1} \circ \Phi_2-\psi \mathbf{1}_{[0,z]}  \right|\text{ d}\lambda \\
&\leq 2\|\varphi\|_{BV} \sup_{0 \leq z \leq 1} \int |\psi| |\mathbf{1}_{[0,z]} \circ \Phi_1^{-1} \circ \Phi_2-\mathbf{1}_{[0,z]}|\text{ d}\lambda \\
&= 2\|\varphi\|_{BV} \sup_{0 \leq z \leq 1} \int |\mathbf{1}_{[0,z]} \circ \Phi_1^{-1} -\mathbf{1}_{[0,z]} \circ \Phi_2^{-1}| \cdot \frac{1}{\Phi_2'\circ \Phi_2^{-1}}\text{ d}\lambda \\
& \leq \frac{2}{1-\varepsilon}\|\varphi\|_{BV} \sup_{0 \leq z \leq 1} \int |\mathbf{1}_{[\Phi_1(0),\Phi_1(z)]}-\mathbf{1}_{[\Phi_2(0),\Phi_2(z)]}|\text{ d}\lambda \\
& \leq \frac{2\|\varphi\|_{BV}}{1-\varepsilon} (|\Phi_1(0)-\Phi_2(0)| +\max_{0 \leq t \leq 1}|\Phi_1(t)-\Phi_2(t)|) \\
&\leq \frac{2\|\varphi\|_{BV}}{1-\varepsilon} 2\max_{0 \leq t \leq 1}|\Phi_1(t)-\Phi_2(t)| \\
& = \frac{4\|\varphi\|_{BV}}{1-\varepsilon} \max_{0 \leq t \leq 1} \varepsilon \int g(y-t)(f_1(y)-f_2(y))\text{ d}y \leq \frac{2\varepsilon}{1-\varepsilon}\|\varphi\|_{BV}\|f_1-f_2\|_1.
\end{align*}
\end{proof}
The final lemma, a corollary to the Arzel\`a-Ascoli theorem, is a folklore result on function spaces.
\begin{lem}
The space $\mathcal{C}_{R,S}^c$ is a compact, convex subset of
$C^0$ and a fortiori also of
$L^1$.
\end{lem}
By Schauder's fixed point theorem we can conclude that there exists a fixed point of $\mathcal{F}_{\varepsilon}$ in $\mathcal{C}_{R,S}^c$. This completes our proof of Proposition \ref{propfix}.
\end{proof}

\subsection{Convergence to the invariant density} We prove the following proposition in this section:
\begin{prop} \label{propconv}
Let $R>0$, $S>0$ and $c>0$.
Then there exist $C>0$, $\gamma\in(0,1)$ and an $\varepsilon^*(R,S,c)>0$ such that for $0 \leq \varepsilon < \varepsilon^*(R,S,c)$
\[
\|\mathcal{F}_{\varepsilon}^n(f_0)-f_*^{\varepsilon}\|_{BV} \leq C \gamma^n \|f_0-f_*^{\varepsilon}\|_{BV}\quad\text{for all }f_0 \in \mathcal{C}_{R,S}^c\text{ and }n\in\mathbb{N}.
\]
\end{prop}
\begin{proof}
It is obviously enough to prove this proposition for sufficiently large $R,S$ and $c$.
In particular we can assume that $R>R^*,S>S^*,c>c^*$ and also $\varepsilon<\varepsilon_0^*$, where $R^*,S^*,c^*,\varepsilon_0^*$
are chosen as in Lemma~\ref{lem:1}

The following proof is strongly inspired by the proof of \cite[Theorem 4]{{keller1982stochastic}}. We start by proving a lemma similar to \cite[Lemma 8]{{keller1982stochastic}}. 

\begin{lem} \label{lemma8}
There exist $0 < \varepsilon_1^* < \varepsilon^*_0$, $\beta<1$ and $c_1>0$ such that for $0 \leq \varepsilon < \varepsilon_1^*$
\[
\| P_TP_{\Phi_{n}}\dots P_TP_{\Phi_{1}}u\|_{BV} \leq c_1 \cdot \beta^n \|u\|_{BV}
\]
for any $\Phi_{i}=\Phi_{f_i}$ for which $f_i \in \mathcal{C}_{R,S}^c$,
any
$u \in BV_{0}=\{v \in BV, \text{ } \int v\text{ d}\lambda=0 \}$ and any $n \in \mathbb{N}$.
\end{lem}

\begin{proof}
Let $\mathcal{Q}_{\varepsilon,n}=P_TP_{\Phi_{n}}\dots P_TP_{\Phi_{1}}$, and let $\Phi_*$ be the coupling function associated to the invariant density $f_*^{\varepsilon}$. The lemma would be immediate if $f_i=f_*^{\varepsilon}$ would hold for all of the densities. So what we need to show is that $\mathcal{Q}_{\varepsilon,n}$ is close to $(P_TP_{\Phi_*})^n$ in a suitable sense.

The proof has three main ingredients: we first show that $\mathcal{Q}_{\varepsilon,n}: BV_0 \to BV_0$ is a uniformly bounded operator. The second fact we are going to see is that $(P_TP_{\Phi_*})^N: BV_0 \to L^1$ is a bounded operator and the operator norm can be made suitably small by choosing $N$ large enough. Lastly, referring to Claim~\ref{l1} we argue that the norm of $P_{\Phi_k}-P_{\Phi_*}: BV_0 \to L^1$ is of order $\varepsilon$. These three facts will imply that $\mathcal{Q}_{\varepsilon,n}$ is a contraction on $BV_0$ for $n$ large enough.

Remember that $P_T P_{\Phi_k}=P_{F_k}$, where $F_k$ is a $C^2$ expanding map of $\mathbb{T}$. Hence it is mixing and satisfies a Lasota-Yorke type inequality. More precisely,
\begin{align} \label{ly}
\|P_T P_{\Phi_k} u\|_{BV} &= \|P_{F_k} u\|_{BV} \leq \|u\|_1+\frac{1}{\inf|F_k'|}\text{var}(u)+\max_{ i=1,\dots,N}\sup \frac{|(F^{-1,i}_k)''|}{|(F^{-1,i}_k)'|}\|u\|_1 \nonumber \\
& \leq \frac{1}{\omega(1-\varepsilon)}\|u\|_{BV}+\left(1+\tilde{D} \right)\|u\|_1.
\end{align}
where $\tilde{D}=\max_{ i=1,\dots,N}\sup \frac{|(F^{-1,i}_k)''|}{|(F^{-1,i}_k)'|}$. Let
\begin{equation}\label{eq:GK1}
\alpha=\frac{1}{\omega(1-\varepsilon_0^*)} \quad \text{and} \quad K_0=\frac{1+\tilde{D}}{1-\alpha}.
\end{equation}
Our assumptions on $\varepsilon_0^*$ already provide that $\omega(1-\varepsilon_0^*) > 1$, implying $\alpha < 1$. Then by applying \eqref{ly} repeatedly we get
\begin{equation}\label{eq:GK2}
\|\mathcal{Q}_{\varepsilon,n} u\|_{BV} \leq \alpha^n \|u\|_{BV}+K_0\|u\|_1 \leq (\alpha^n+K_0)\|u\|_{BV}.
\end{equation}
Hence
\begin{equation} \label{unif}
\|\mathcal{Q}_{\varepsilon,n}\|_{BV} \leq K_0+1.
\end{equation}
 A consequence of the fact that $F_{f_*^{\varepsilon}}$ is mixing and $P_{f_*^{\varepsilon}}=P_TP_{\Phi_*}$ satisfies \eqref{ly} is that the spectrum of $P_TP_{\Phi_*}$ consists of the simple eigenvalue 1 and a part contained in a disc of radius $r < 1$ (see e.g.~\cite{baladi2000positive}). From this it follows that there exists an $N \in \mathbb{N}$ such that
\begin{equation}\label{eq:GK3}
\|(P_TP_{\Phi_*})^Nu\|_1 \leq \frac{1}{8K_0}\|u\|_{BV} \quad \forall u \in BV_0.
\end{equation}
Choose $N$ larger if necessary so that we also have
\[
\alpha^N(K_0+1) < \frac{1}{4}.
\]
Now according to Claim~\ref{l1} from the proof of Lemma~\ref{lemma_cont} we have
\[
\|(P_{\Phi_k}-P_{\Phi_*})u\|_1 \leq \text{const} \cdot \varepsilon \|u\|_{BV}\|f_k-f_*\|_1,
\]
where $f_k$ is the density corresponding to $\Phi_k$. Using this,
\begin{align*}
|\|\mathcal{Q}_{\varepsilon,N} u\|_1-\|(P_TP_{\Phi_{*}})^Nu\|_1| &\leq \sum_{j=1}^N \|P_TP_{\Phi_N}\dots P_TP_{\Phi_{j+1}}P_T(P_{\Phi_{j}}-P_{\Phi_{*}})(P_TP_{\Phi_{*}})^{j-1}u\|_1 \\
& \leq \text{const}_N \varepsilon \|u\|_{BV}.
\end{align*}
Combined with \eqref{eq:GK2} -- \eqref{eq:GK3} this implies for sufficiently small $\varepsilon_1^*$ (depending on $N$) and all $\varepsilon\in[0,\varepsilon_1^*)$
\begin{align*}
\|\mathcal{Q}_{\varepsilon,2N} u\|_{BV} &\leq \alpha^N\|\mathcal{Q}_{\varepsilon,N} u\|_{BV}+K_0\|\mathcal{Q}_{\varepsilon,N} u\|_1 \\
& \leq \alpha^N(K_0+1)\|u\|_{BV}+K_0\left(\frac{1}{8K_0}+\text{const}_N \varepsilon \right)\|u\|_{BV} \\
& \leq \frac{1}{2}\|u\|_{BV}.
\end{align*}
The lemma follows if one observes that $\mathcal{Q}_{\varepsilon,n}(BV_0) \subseteq BV_0$ and $\|\mathcal{Q}_{\varepsilon,n}\|_{BV} \leq K_0+1$ for all $n$ by \eqref{unif}.
\end{proof}

Now we move on to the proof of Proposition \ref{propconv}. We remind the reader of the notation $f_{n}=\mathcal{F}_{\varepsilon}^n(f_0)$. Write
\begin{align*}
f_{n+1}-f_*^{\varepsilon}&=P_TP_{\Phi_n}(f_n-f_*^{\varepsilon})+P_T(P_{\Phi_n}-P_{\Phi_*})f_*^{\varepsilon} \\
&\hspace{0.22cm}\vdots \\
&=P_TP_{\Phi_n}\dots P_TP_{\Phi_0}(f_0-f_*^{\varepsilon})+\sum_{k=0}^n P_TP_{\Phi_n}\dots P_TP_{\Phi_{k+1}}P_T(P_{\Phi_k}-P_{\Phi_*})f_*^{\varepsilon}
\end{align*}
Then Lemma \ref{lemma8} implies that
\begin{equation}\label{eq:GK4}
\|f_{n+1}-f_*^{\varepsilon}\|_{BV} \leq c_1 \beta^{n+1}\|f_0-f_*^{\varepsilon}\|_{BV}+c_1 \sum_{k=0}^n \beta^{n-k}\|P_T(P_{\Phi_k}-P_{\Phi_*})f_*^{\varepsilon}\|_{BV}.
\end{equation}

\begin{claim} \label{claim2} Let $\varphi \in \mathcal{C}_{R,S}^c$ and  $f_1,f_2 \in \tilde{\mathcal{C}}_{R,S}^c=\{ \text{var}(f) \leq R, |f'| \leq S, \text{Lip}(f') \leq c \}$. Denote $\Phi_{f_1}=\Phi_1$, $\Phi_{f_2}=\Phi_2$. Then
\[
\|(P_{\Phi_1}-P_{\Phi_2})\varphi\|_{BV} \leq \varepsilon K(R,S,c) \|f_1-f_2\|_{BV}
\]
for some constant $K=K(R,S,c)$.
\end{claim}

Suppose Claim~\ref{claim2} holds (we are going to prove it later). This implies that
\begin{align*}
\|P_T(P_{\Phi_k}-P_{\Phi_*})f_*^{\varepsilon}\|_{BV} \leq \varepsilon K\|P_T\|_{BV}\|f_k-f_*^{\varepsilon}\|_{BV}
\end{align*}

Choose $\gamma\in(\beta,1)$ and $C>c_1$ where $\beta$ and $c_1$ are the constants from \eqref{eq:GK4}. Then, by using induction, we get
\begin{align*}
\|f_{n+1}-f_*^{\varepsilon}\|_{BV} &\leq c_1 \beta^{n+1}\|f_0-f_*^{\varepsilon}\|_{BV}+c_1\varepsilon K(R,S,c) \|P_T\|_{BV} \sum_{k=0}^n \beta^{n-k} \|f_k-f_*^{\varepsilon}\|_{BV} \\
& \leq c_1 \beta^{n+1}\|f_0-f_*^{\varepsilon}\|_{BV}+\varepsilon c_2(R,S,c) \sum_{k=0}^n \beta^{n-k} C\gamma^k \|f_0-f_*^{\varepsilon}\|_{BV} \\
& \leq c_1 \beta^{n+1}\|f_0-f_*^{\varepsilon}\|_{BV}+\varepsilon c_2(R,S,c) C \beta^n \|f_0-f_*^{\varepsilon}\|_{BV} \sum_{k=0}^n \left( \frac{\gamma}{\beta}\right)^k \\
& \leq c_1 \gamma^{n+1}\|f_0-f_*^{\varepsilon}\|_{BV}+\varepsilon c_3(R,S,c) C \gamma^n \|f_0-f_*^{\varepsilon}\|_{BV},
\end{align*}
so if we choose $\varepsilon^*(R,S,c):=\min\left\{\varepsilon_1^*, \frac{\gamma(C-c_1)}{c_3(R,S,c)C}\right\}$, then
\[
\|f_{n+1}-f_*^{\varepsilon}\|_{BV} \leq C \gamma^{n+1} \|f_0-f_*^{\varepsilon}\|_{BV}.
\]
This concludes the proof of Proposition~\ref{propconv} and thus the proof of Theorem~\ref{theomain}. What is left is the proof of Claim~\ref{claim2}.

\emph{Proof of Claim~\ref{claim2}.} First note that as $\int (P_{\Phi_1}-P_{\Phi_2})\varphi\text{ d}\lambda=0$,
\[
\|(P_{\Phi_1}-P_{\Phi_2})\varphi\|_{BV} \leq \frac{3}{2} \var((P_{\Phi_1}-P_{\Phi_2})\varphi),
\]
so we only need to give the appropriate bound on the total variation.
\begin{align*}
\text{var}((P_{\Phi_1}-P_{\Phi_2})\varphi)&=\int \left |
\left( \frac{\varphi}{\Phi'_1} \circ \Phi^{-1}_1-\frac{\varphi}{\Phi'_2} \circ \Phi^{-1}_2 \right)' \right|\text{ d}\lambda
\\
&=\int \left | \frac{\varphi'}{(\Phi'_1)^2} \circ \Phi^{-1}_1-\frac{\varphi'}{(\Phi'_2)^2} \circ \Phi^{-1}_2 +\frac{\varphi \cdot \Phi_2''}{(\Phi_2')^3} \circ \Phi_2^{-1}-\frac{\varphi \cdot \Phi_1''}{(\Phi_1')^3} \circ \Phi_1^{-1}  \right|\text{ d}\lambda \\
& \leq \underbrace{\int \left | \frac{\varphi'}{(\Phi'_1)^2} \circ \Phi^{-1}_1-\frac{\varphi'}{(\Phi'_2)^2} \circ \Phi^{-1}_2 \right|\text{ d}\lambda}_{(C)} + \underbrace{\int \left | \frac{\varphi \cdot \Phi_2''}{(\Phi_2')^3} \circ \Phi_2^{-1}-\frac{\varphi \cdot \Phi_1''}{(\Phi_1')^3} \circ \Phi_1^{-1}  \right|\text{ d}\lambda}_{(D)}
\end{align*}
We start by giving a bound for the term $(C)$.
\[
(C) \leq \int \underbrace{\left | \frac{\varphi'}{(\Phi'_1)^2} \circ \Phi^{-1}_1-\frac{\varphi'}{(\Phi'_2)^2} \circ \Phi^{-1}_1 \right|\text{ d}\lambda}_{(C1)}+\underbrace{\int \left | \frac{\varphi'}{(\Phi'_2)^2} \circ \Phi^{-1}_1-\frac{\varphi'}{(\Phi'_2)^2} \circ \Phi^{-1}_2 \right|\text{ d}\lambda}_{(C2)}
\]
As for $(C1)$,
\begin{align*}
(C1) &=\int |\varphi'| \cdot \left|\frac{1}{\Phi_1'}-\frac{\Phi_1'}{(\Phi_2')^2} \right|\text{ d}\lambda=\int |\varphi'| \cdot \left|\frac{(\Phi_2')^2-(\Phi_1')^2}{(\Phi_2')^2\Phi_1'} \right|\text{ d}\lambda \\
&\leq \max \left| \frac{\varphi'(\Phi_1'+\Phi_2')}{(\Phi_2')^2 \Phi_1'}\right|\int |\Phi_1'-\Phi_2'|\text{ d}\lambda
 \leq \frac{2S(1+\varepsilon R)}{(1-\varepsilon R)^3} \cdot \varepsilon \|f_1-f_2\|_1 \leq \varepsilon K_{C1}\|f_1-f_2\|_{BV}.
\end{align*}
Note that $1-\varepsilon R > 0$ by our choice of $R$ and $\varepsilon$. $(C2)$ can be bounded the same way as term (A) in the proof of Claim~\ref{l1}, so we have
\begin{align*}
(C2) & \leq \left \|\frac{\varphi'}{\Phi_2'} \right \|_{BV} \frac{2\varepsilon}{1-\varepsilon R}\|f_1-f_2\|_1 \\
&\leq \left(\left \|\frac{\varphi'}{\Phi_2'} \right \|_{1}+\sup|\varphi'|\var\left(\frac{1}{\Phi_2'}\right)+\var|\varphi'|\sup\left(\frac{1}{\Phi_2'}\right) \right)  \frac{2\varepsilon}{1-\varepsilon R}\|f_1-f_2\|_1 \\
& \leq \frac{2\varepsilon(R+S^2\varepsilon/(1-\varepsilon)+c)}{1-\varepsilon R}\|f_1-f_2\|_1 \leq \varepsilon K_{C2}\|f_1-f_2\|_{BV}.
\end{align*}
We now move on to bounding $(D)$.
\begin{align*}
(D) &\leq \underbrace{\int \left | \frac{\varphi \cdot \Phi_2''}{(\Phi_2')^3} \circ \Phi_2^{-1}-\frac{\varphi \cdot \Phi_1''}{(\Phi_1')^3} \circ \Phi_2^{-1}  \right|\text{ d}\lambda}_{(D1)}+\underbrace{\int \left | \frac{\varphi \cdot \Phi_1''}{(\Phi_1')^3} \circ \Phi_2^{-1}-\frac{\varphi \cdot \Phi_1''}{(\Phi_1')^3} \circ \Phi_1^{-1}  \right|\text{ d}\lambda}_{(D2)}
\end{align*}
We start with $(D1)$.
\begin{align*}
(D1)&=\int \left | \frac{\varphi \cdot \Phi_2''}{(\Phi_2')^2}-\frac{\varphi \cdot \Phi_1''}{(\Phi_1')^3} \cdot \Phi_2'  \right|\text{ d}\lambda=\int |\varphi(y)| \cdot \left| \frac{\varepsilon f_2' \left(y+\frac{1}{2}\right)}{[\Phi_2'(y)]^2}-\frac{\varepsilon f_1' \left(y+\frac{1}{2}\right)\Phi_2'(y)}{[\Phi_1'(y)]^3} \right|\text{ d}y  \\
&\leq \varepsilon \max \left|\frac{\varphi}{(\Phi_2')^2(\Phi_1')^3} \right|\int |f_2'\left(y+\frac{1}{2}\right)[\Phi_1'(y)]^3-f_1'\left(y+\frac{1}{2}\right)[\Phi_2'(y)]^3|\text{ d}y \\
& \leq \frac{\varepsilon(1+R)}{(1-\varepsilon R)^5}\bigg(\underbrace{\int|f_2'\left(y+\frac{1}{2}\right)([\Phi_1'(y)]^3-[\Phi_2'(y)]^3)|\text{ d}y}_{(D11)}\\
&+\underbrace{\int |[f_1'\left(y+\frac{1}{2}\right)-f_2'\left(y+\frac{1}{2}\right)][\Phi_2'(y)]^3|\text{ d}y}_{(D12)} \bigg)
\end{align*}
We can bound $(D11)$ and $(D12)$ in the following way:
\begin{align*}
(D11) & \leq \max|f_2'((\Phi_1')^2-2\Phi_1'\Phi_2'+(\Phi_2')^2)| \int |\Phi_1'-\Phi_2'|\text{ d}\lambda \leq 4S(1+\varepsilon R)^2 \varepsilon \|f_1-f_2\|_1 \\
(D12) & \leq \max|(\Phi_2')^3|\var(f_1-f_2) \leq (1+\varepsilon R)^3 \|f_1-f_2\|_{BV}
\end{align*}
Summarizing the bound for $(D1)$ we see that
\[
(D1) \leq \varepsilon K_{D1}\|f_1-f_2\|_{BV}.
\]
The last term left to bound is $(D2)$. This term can be bounded as term (A) in the proof of Claim~\ref{l1}, so we have
\begin{align*}
(D2) &\leq 2\left \| \frac{\varphi \cdot \Phi_1''}{(\Phi_1')^2} \right \|_{BV} \frac{\varepsilon}{1-\varepsilon R} \|f_1-f_2\|_1 \leq 2\left ( \max \left| \frac{\varphi \cdot \Phi_1''}{(\Phi_1')^2} \right|+ \text{Lip} \left( \frac{\varphi \cdot \Phi_1''}{(\Phi_1')^2} \right) \right ) \frac{\varepsilon}{1-\varepsilon R} \|f_1-f_2\|_1 \\
& \leq 2 \left ( \frac{\varepsilon (1+R)S}{(1-\varepsilon R)^2}+ \varepsilon K(R,S) \right ) \frac{\varepsilon}{1-\varepsilon R} \|f_1-f_2\|_1 \\
&\leq \varepsilon K_{D2}\|f_1-f_2\|_{BV}
\end{align*}
In conclusion,
\begin{align*}
\|(P_{\Phi_1}-P_{\Phi_2})\varphi\|_{BV} &\leq \varepsilon \cdot \frac{3}{2}(K_{C1}+K_{C2}+K_{D1}+K_{D2})\|f_1-f_2\|_{BV}.
\end{align*}
\end{proof}

\section{Proof of Theorem~\ref{stability}} \label{s:theo2proof}

We remind the reader of Theorem~\ref{stability}: It states that when $R,S,c$ and $\varepsilon^*$ are chosen as in Theorem~\ref{theomain}, then for all $0 \leq \varepsilon,\varepsilon' < \varepsilon^{*}$ we have
\[
\|f_*^{\varepsilon}-f_*^{\varepsilon'}\|_{BV} \leq K|\varepsilon-\varepsilon'|,
\]
for some $K>0$, where $f_*^{\varepsilon}$ and $f_*^{\varepsilon'}$ are the fixed densities of $\mathcal{F}_{\varepsilon}$ and $\mathcal{F}_{\varepsilon'}$, respectively.

We now proceed to prove this. First observe that the choice of $R,S$ and $c$ implies that the fixed density of $\mathcal{F}_{\varepsilon}$ is in $\mathcal{C}_{R,S}^c$ for all $0 \leq \varepsilon < \varepsilon^*$. In particular, $f_*^{\varepsilon'} \in \mathcal{C}_{R,S}^c$.

Let $f$ be an element of $\mathcal{C}_{R,S}^c$. It is a consequence of Theorem~\ref{theomain} that there exists an $N \in \mathbb{N}$ such that
\[
\|\mathcal{F}_{\varepsilon'}^Nf-f_*^{\varepsilon'}\|_{BV} \leq \lambda\|f-f_*^{\varepsilon'}\|_{BV}
\]
for some $0 < \lambda < 1$.

This implies that
\begin{align*}
\|\mathcal{F}_{\varepsilon}^Nf-f_*^{\varepsilon'}\|_{BV} &\leq \|\mathcal{F}_{\varepsilon}^Nf-\mathcal{F}_{\varepsilon'}^N f\|_{BV}+\|\mathcal{F}_{\varepsilon'}^Nf-f_*^{\varepsilon'}\|_{BV} \\
& \leq \|\mathcal{F}_{\varepsilon}^Nf-\mathcal{F}_{\varepsilon'}^N f\|_{BV}+\lambda\|f-f_*^{\varepsilon'}\|_{BV}.
\end{align*}

\begin{lem} \label{iterN_general}
For each $N \in \mathbb{N}$ there exists an $a_N > 0$ such that if $f \in \mathcal{C}_{R,S}^c$ and $0 \leq \varepsilon , \varepsilon' < \varepsilon^*$, then $\|\mathcal{F}_{\varepsilon}^Nf-\mathcal{F}_{\varepsilon'}^N f\|_{BV} \leq a_N |\varepsilon-\varepsilon'|$.
\end{lem}
We are going to give the proof of this lemma later. Using this result we have that
\[
\|\mathcal{F}_{\varepsilon}^Nf-f^{\varepsilon'}_*\|_{BV} \leq a_N |\varepsilon-\varepsilon'|+\lambda\|f-f_*^{\varepsilon'}\|_{BV}.
\]
Let $B(h,r)$ denote the ball in $BV$ with center $h$ and radius $r$. We claim that $\mathcal{F}_{\varepsilon}^N$ leaves the ball $B(f_*^{\varepsilon'},\bar{r})$ invariant, where $\bar{r}=\frac{a_N |\varepsilon-\varepsilon'|}{1-\lambda}$. To see this, suppose that $\|f-f_*^{\varepsilon'}\|_{BV} \leq \bar{r}$. Then indeed,
\[
\|\mathcal{F}_{\varepsilon}^Nf-f_*^{\varepsilon'}\|_{BV} \leq a_N|\varepsilon-\varepsilon'|+\lambda \bar{r}=a_N|\varepsilon-\varepsilon'|+\frac{\lambda a_N|\varepsilon-\varepsilon'|}{1-\lambda}=\frac{a_N|\varepsilon-\varepsilon'|}{1-\lambda}=\bar{r}.
\]
Note that Theorem~\ref{theomain} implies that
$\lim_{k\to\infty}\|\mathcal{F}^{kN}(f_0)-f^\varepsilon_*\|_{BV}=0$ for all $f_0\in C^c_{R,S}$. Note further that $B(f_*^{\varepsilon'},\bar{r}) \cap \mathcal{C}_{R,S}^c \neq \emptyset$. Hence $f_*^{\varepsilon} \in B(f_*^{\varepsilon'},\bar{r})$. This means that
\[
\|f_*^{\varepsilon}-f_*^{\varepsilon'}\|_{BV} \leq \frac{a_N |\varepsilon-\varepsilon'|}{1-\lambda} \eqqcolon K|\varepsilon-\varepsilon'|,
\]
which is exactly the statement to prove, so we are only left to prove Lemma~\ref{iterN_general}.

\emph{Proof of Lemma~\ref{iterN_general}.} Suppose (without loss of generality) that $0 \leq \varepsilon' \leq \varepsilon$. Notice that if we use the notation
\[
\Phi^{\varepsilon}_f(x)=x+\varepsilon \int_{\mathbb{T}} g(y-x)f(y)\text{ d}y
\]
then
\[
\Phi^{\varepsilon'}_f=\Phi^{\varepsilon}_{\varepsilon'f/\varepsilon}.
\]
By Claim~\ref{claim2},
\begin{equation} \label{lenyeg_general}
\|(P_{\Phi_f^{\varepsilon}}-P_{\Phi_{\varepsilon'f/\varepsilon}^{\varepsilon}})\varphi\|_{BV} \leq K_1\varepsilon\|f-\varepsilon'f/\varepsilon\|_{BV}=K_1 |\varepsilon-\varepsilon'|\|f\|_{BV}
\end{equation}
holds. (Here we have used the fact that $\varepsilon'/\varepsilon \leq 1$, so $\varepsilon'f/\varepsilon \in \mathcal{C}^c_{R,S}$, hence Claim~\ref{claim2} can be applied indeed.) $P_T$ is a bounded operator on $BV$, let $\|P_T\|_{BV} \leq K_2$. This implies that
\begin{align} \label{N1_general}
\|\mathcal{F}_{\varepsilon}f-\mathcal{F}_{\varepsilon'}f\|_{BV} 
&=
\|P_TP_{\Phi_f^\varepsilon}f-P_TP_{\Phi_{\varepsilon'f/\varepsilon}^\varepsilon}f\|_{BV} \leq \bar K|\varepsilon-\varepsilon'|\|f\|_{BV}, 
\end{align}
where $\bar{K}=K_1K_2$.

Now we prove the lemma using induction. Since $\|f\|_{BV} \leq 1+R$, the case $N=1$ holds by \eqref{N1_general} with the choice of $a_1=\bar{K}(1+R)$. Assume that
\begin{equation} \label{Nmin1_general}
\|\mathcal{F}_{\varepsilon}^{N-1}f-\mathcal{F}_{\varepsilon'}^{N-1}f\|_{BV} \leq a_{N-1}|\varepsilon-\varepsilon'|.
\end{equation}
Then using the $N=1$ case
\begin{align*}
\|\mathcal{F}_{\varepsilon}^{N}f-\mathcal{F}_{\varepsilon'}^{N}f\|_{BV}
&\leq
\|\mathcal{F}_{\varepsilon}\mathcal{F}_{\varepsilon}^{N-1}f-\mathcal{F}_{\varepsilon'}\mathcal{F}_{\varepsilon}^{N-1}f\|_{BV}+\|\mathcal{F}_{\varepsilon'}\mathcal{F}_{\varepsilon}^{N-1}f-\mathcal{F}_{\varepsilon'}\mathcal{F}_{\varepsilon'}^{N-1}f\|_{BV} \\
& \leq a_1 |\varepsilon-\varepsilon'| + \|\mathcal{F}_{\varepsilon'}\mathcal{F}_{\varepsilon}^{N-1}f-\mathcal{F}_{\varepsilon'}\mathcal{F}_{\varepsilon'}^{N-1}f\|_{BV}.
\end{align*}
Let $\varphi=\mathcal{F}_{\varepsilon}^{N-1}f$ and $\psi=\mathcal{F}_{\varepsilon'}^{N-1}f$. Then
\begin{align*}
\|\mathcal{F}_{\varepsilon'}\mathcal{F}_{\varepsilon}^{N-1}f-\mathcal{F}_{\varepsilon'}\mathcal{F}_{\varepsilon'}^{N-1}f\|_{BV}
&=\|P_TP_{\Phi_{\varphi}^{\varepsilon'}}\varphi-P_TP_{\Phi_{\psi}^{\varepsilon'}}\psi\|_{BV} \\
& \leq \|P_TP_{\Phi_{\varphi}^{\varepsilon'}}(\varphi-\psi)\|_{BV}+\|P_T(P_{\Phi_{\varphi}^{\varepsilon'}}-P_{\Phi_{\psi}^{\varepsilon'}})\psi\|_{BV} \\
& \leq  (c_1 \cdot \beta+\varepsilon' \bar{K}) \|\varphi-\psi\|_{BV}
\end{align*}
by using Lemma~\ref{lemma8} with $n=1$ for the first term and Claim~\ref{claim2} for the second term in line two. Let $\bar{A}= c_1 \cdot \beta+\varepsilon^* \bar{K}$. Now we can proceed in the following way:
\begin{align*}
\|\mathcal{F}_{\varepsilon}^{N}f-\mathcal{F}_{\varepsilon'}^{N}f\|_{BV}
& \leq a_1 |\varepsilon-\varepsilon'| + \bar{A} \|\mathcal{F}_{\varepsilon}^{N-1}f-\mathcal{F}_{\varepsilon'}^{N-1}f\|_{BV} \\
& \leq a_1 |\varepsilon-\varepsilon'| + \bar{A} a_{N-1} |\varepsilon-\varepsilon'|
\end{align*}
by using \eqref{Nmin1_general}. So $a_N=a_1+ \bar{A} a_{N-1}$ is an appropriate choice.

Now Lemma~\ref{iterN_general} and thus Theorem~\ref{stability} are proved.

\section{Proof of Theorem \ref{largeeps}}
\label{s:theo3proof}

We start by proving the statement about the support of our initial density shrinking to a single point. Remember that we denoted the smallest closed interval containing the support of $f$ by $\supp^*(f)$. The assumption (F') implies that this interval has length less than 1/2.

\begin{prop} \label{propshrink}
Suppose $f_0$ is of property (F'). Let $\Omega=\max |T'|$, $f_n=\mathcal{F}_{\varepsilon}^n(f_0)$. Then
\[
1-\frac{1}{\Omega} < \varepsilon < 1 \qquad \Rightarrow \qquad |\supp^* (f_n)| \underset{n \to \infty}{\to} 0
\]
\end{prop}

\begin{proof}
We are going to denote the length of an interval $I$ by $|I|$. First notice that $\supp^*(f_1)=F_{f_0}(\supp^*(f_0))$. Since $F_{f_0}$ is $C^2$, we have that
\begin{align*}
|\supp^*(f_1)|=|F_{f_0}(\supp^*(f_0))| &\leq \max_{x \in \supp^*(f_0)}|F'_{f_0}(x)||\supp^*(f_0)|, \\
&\leq \max|T'| \max_{x \in \supp^*(f_0)}|\Phi_{f_0}'(x)||\supp^*(f_0)|, \\
&\leq \Omega  \max_{x \in \supp^*(f_0)}|\Phi_{f_0}'(x)||\supp^*(f_0)|, \\
& \leq \Omega(1-\varepsilon)|\supp^*(f_0)|,
\end{align*}
since $\Phi'_{f_0}(x)=1+\varepsilon\left(f_0(x+1/2)-1 \right)=1-\varepsilon$ for $x \in \supp^*(f_0)$.

Let $q=\Omega(1-\varepsilon)$. By this notation we have
\[
|\supp^*(f_1)| \leq q |\supp^*(f_0)|,
\]
and by iteration we get
\[
|\supp^*(f_n)| \leq q^n |\supp^*(f_0)|,
\]
implying $|\supp^* (f_n)| \underset{n \to \infty}{\to} 0$ if $q<1$, which holds if $1-\frac{1}{\Omega} < \varepsilon$.
\end{proof}

\begin{rem}
We required that the support of $f_0$ should fit in an interval with length less then 1/2. The significance of 1/2 comes from the fact that the distance function $g$ used in the coupling $\Phi$ has singularities at $\pm 1/2$. This results in the fact that $f_n(x+1/2)$ plays a role in $\Phi_{f_n}'(x+1/2)$.

However, singularities in the function $g$ are not vital to this phenomenon. It can be shown for some special continuous functions $g$ that if the support of the initial density $f_0$ is small in terms of some properties of $g$ (bounds on the supremum or the derivative), the supports of the densities $f_n$ shrink exponentially.
\end{rem}
Since the corresponding measures $\mu_n$ are all probability measures, this proposition implies that if the sequence $(\mu_n)$ converges, it can only converge to a Dirac measure supported on some $x^* \in \mathbb{T}$. But typically this is not the case. What happens typically is that the sequence $\mu_n$ is in fact divergent and approaches a Dirac measure moving along the $T-$trajectory of some $x^* \in \supp^*(f_0)$. More precisely:

\begin{prop} \label{wasser}
Suppose $f_0$ is of property (F') and $1-\frac{1}{\Omega} < \varepsilon$. Then there exists an $x^* \in \supp^*(f_0)$ such that
\[
\lim_{n \to \infty}W_1(\mu_n,\delta_{T^n(x^*)})=0,
\]
where $W_1(\cdot,\cdot)$ is the Wasserstein metric.
\end{prop}

\begin{proof} We start by proving a short lemma.
\begin{lem}
We claim that $\supp^* P_{\Phi_{f}}f \subset \supp^* f$.
\end{lem}
\begin{proof}
Consider the lifted coupling dynamics $\Phi_{f}^{\mathbb{R}}: \mathbb{R} \to \mathbb{R}$ which is defined as
\[
\Phi_{f}^{\mathbb{R}}(x)=x+\varepsilon \int_0^1 g(y-x)f(y)\text{d}y \qquad x \in \mathbb{R}.
\]

We have to study two cases separately. First, let $\supp^*f=[a,b]$, $0 \leq a < b \leq 1$.  In this case, notice that
\[
\int_0^1 yf(y) \text{ d}y=M \in [a,b].
\]
Now
\begin{align*}
\Phi_{f}^{\mathbb{R}}(a)&=a+\varepsilon\int_0^1(y-a)f(y)\text{ d}y =a+\varepsilon(M-a), \\
\Phi_{f}^{\mathbb{R}}(b)&=b+\varepsilon\int_0^1(y-b)f(y)\text{ d}y =b-\varepsilon(b-M).
\end{align*}
Hence $[\Phi_{f}^{\mathbb{R}}(a), \Phi_{f}^{\mathbb{R}}(b)] \subset [a,b]$, and this implies $\supp^* P_{\Phi_{f}}f \subset \supp^* f$.

In the second case, $\supp^*f=[a,1]\cup[0,b]$, $0 < b < a < 1$. Now
\[
\int_0^b (y+1)f(y) \text{ d}y+\int_a^1 yf(y) \text{ d}y=\tilde{M} \in [a,1+b].
\]
On the one hand, we see that
\begin{align*}
\Phi_{f}^{\mathbb{R}}(a)&=a+\varepsilon\left(\int_0^b g(y-a)f(y) \text{ d}y+\int_a^1 g(y-a)f(y) \text{ d}y \right), \\
&= a+\varepsilon\left(\int_0^b (y-a+1)f(y) \text{ d}y+\int_a^1 (y-a)f(y) \text{ d}y \right), \\
&=a+\varepsilon(\tilde{M}-a),
\end{align*}
and
\begin{align*}
\Phi_{f}^{\mathbb{R}}(a)&\leq a+\int_0^b (y-a+1)f(y) \text{ d}y+\int_a^1 (y-a)f(y) \text{ d}y,  \\
&=\tilde{M}.
\end{align*}
Furthermore,
\begin{align*}
\Phi_{f}^{\mathbb{R}}(1+b)&=(1+b)+\varepsilon\left(\int_0^b g(y-(1+b))f(y) \text{ d}y+\int_a^1 g(y-(1+b))f(y) \text{ d}y \right), \\
&=(1+b)+\varepsilon\left(\int_0^b (y-(1+b)+1)f(y) \text{ d}y+\int_a^1 (y-(1+b))f(y) \text{ d}y \right), \\
&=(1+b)-\varepsilon((1+b)-\tilde{M}),
\end{align*}
and
\begin{align*}
\Phi_{f}^{\mathbb{R}}(1+b)&  \geq (1+b)-1 \cdot ((1+b)-\tilde{M})=\tilde{M}.
\end{align*}
Hence $[\Phi_{\mu}^{\mathbb{R}}(a), \Phi_{\mu}^{\mathbb{R}}(1+b)] \subset [a,1+b]$, and this implies $\supp^* P_{\Phi_{f}}f \subset \supp^* f$.
\end{proof}

The lemma implies that $\supp^* f_1=F_{f_0}(\supp^* f_0) \subset T(\supp^* f_0)$. Now the $T$-preimage of $\supp^* f_1$ is a finite collection of intervals; consider its component that coincides with $\Phi_{f_0}(\supp^* f_0)$. By the lemma above, this is a closed interval strictly contained in $\supp^* f_0$ (see Figure \ref{xstar} for an illustration). Arguing similarly for $\supp^* f_1$ and $\supp^* f_2$ and so on, we get a nested sequence of intervals with lengths shrinking to zero, since $|\Phi_{f_n}(\supp^* f_n)| \leq (1-\varepsilon)^n|\supp^* f_0|$ -- as calculated during the proof of Proposition \ref{propshrink}. By taking their intersection we get a point $x^*$ for which $T^n(x^*) \in \supp^* f_n$ holds.

\begin{figure}[h!]
\begin{tikzpicture}
\draw [gray!90] (0,0) circle (2cm);

\draw [thick, gray!120, domain=10:70] plot ({2*cos(\x)}, {2*sin(\x)});
\draw [thick, gray!120] ({1.9*cos(10)}, {1.9*sin(10)})-- ({2.1*cos(10)}, {2.1*sin(10)});
\draw [thick, gray!120] ({1.9*cos(70)}, {1.9*sin(70)})-- ({2.1*cos(70)}, {2.1*sin(70)});
\draw [gray!120] ({3*cos(40)}, {3*sin(40)}) node {$\supp^*f_0$};

\draw [thick,domain=20:40] plot ({2*cos(\x)}, {2*sin(\x)});
\draw [thick] ({1.9*cos(20)}, {1.9*sin(20)})-- ({2.1*cos(20)}, {2.1*sin(20)});
\draw [thick] ({1.9*cos(40)}, {1.9*sin(40)})-- ({2.1*cos(40)}, {2.1*sin(40)});
\draw ({5.2*cos(12)}, {5.2*sin(12)}) node {$T^{-1,1}F_{f_0}(\supp^*f_0)=\Phi_{f_0}(\supp^*f_0)$};

\draw [thick,domain=-20:-40] plot ({2*cos(\x)}, {2*sin(\x)});
\draw [thick] ({1.9*cos(-20)}, {1.9*sin(-20)})-- ({2.1*cos(-20)}, {2.1*sin(-20)});
\draw [thick] ({1.9*cos(-40)}, {1.9*sin(-40)})-- ({2.1*cos(-40)}, {2.1*sin(-40)});
\draw ({3.8*cos(-18)}, {3.8*sin(-18)}) node {$T^{-1,2}F_{f_0}(\supp^*f_0)$};

\draw [thick,domain=-200:-160] plot ({2*cos(\x)}, {2*sin(\x)});
\draw [thick] ({1.9*cos(-200)}, {1.9*sin(-200)})-- ({2.1*cos(-200)}, {2.1*sin(-200)});
\draw [thick] ({1.9*cos(-160)}, {1.9*sin(-160)})-- ({2.1*cos(-160)}, {2.1*sin(-160)});
\draw ({3.8*cos(-180)}, {3.8*sin(-180)}) node {$T^{-1,3}F_{f_0}(\supp^*f_0)$};
\end{tikzpicture}
\caption{Construction of $x^*$, first step.} \label{xstar}
\end{figure}
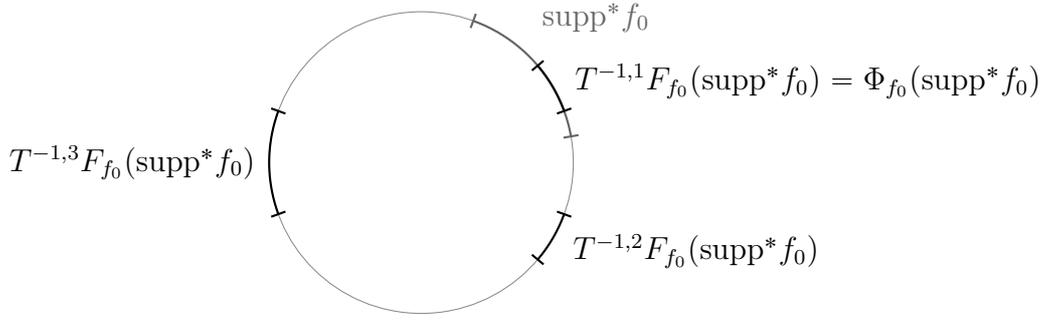

Formally, by using the notation $F_{f_{n-1}}\dots F_{f_0}(\supp^* f_0)=F^n (\supp^* f_0)$,
\[
\{x^*\}=\bigcap_{n=0}^{\infty}T^{-n}F^n(\supp^* f_0).
\]
By the Kantorovich-Rubinstein theorem mentioned in Section \ref{s:results},
\begin{align*}
W_1(\mu_n,\delta_{T^n(x^*)})&=\sup \left|\int_0^1\ell(x)f_n(x) \text{d}x-\ell(T^n(x^*)) \right|,
\end{align*}
where the supremum is taken for all continuous functions $\ell: \mathbb{T} \to \mathbb{R}$ with Lipschitz constant $\leq 1$. Using that $T^n(x^*) \in \supp^* f_n$, we have that
\begin{align*}
\left|\int_0^1\ell(x)f_n(x) \text{d}x-\ell(T^n(x^*)) \right|&=\left|\int_0^1(\ell(x)-\ell(T^n(x^*)) )f_n(x) \text{d}x\right| \leq 1 \cdot |\supp^* f_n| \left| \int_0^1 f_n(x) \text{d}x \right| \\
& \leq [\Omega(1-\varepsilon)]^n.
\end{align*}

This implies $W_1(\mu_n,\delta_{T^n(x^*)}) \to 0$.
\end{proof}

\section{Concluding remarks} \label{s:remarks}

\subsection{Choice of the space $\mathcal{C}_{R,S}^c$}
\label{s:discussion}

As emphasized before, the space $\mathcal{C}_{R,S}^c$ is carefully chosen for the proof to work. In particular, the key statements that quantify the weak-coupling behavior of the transfer operators are Claim~\ref{l1} in the proof of Lemma~\ref{lemma_cont} and Claim~\ref{claim2} in the proof of Proposition~\ref{propconv}. In the proof of Claim~\ref{l1} regularity is not essential, since Lemma~\ref{lemma_cont} is needed in the course of proving the \emph{existence} of the invariant density, and we believe this is true in a much general context. On the other hand, Proposition~\ref{propconv} concerns the \emph{stability} (and in conclusion the \emph{uniqueness}) of the invariant density. This is where the regularity of the densities plays a crucial role, more specifically in the proof of Claim~\ref{claim2}. The uniform bound on the total variation and the derivative of the densities is needed recurrently, while the uniform Lipschitz continuity of the derivative of the densities is essential in bounding the term (C2). This explains the definition of $\mathcal{C}_{R,S}^c$. In turn, we need to assume the smoothness of $T$ to ensure that $P_T$ preserves $\mathcal{C}_{R,S}^c$.

\subsection{Relation to other works and open problems}

The goal of this paper was to show that the results of \cite{selley2016mean} on stability and synchronization in an infinite system can be generalized to a wider class of coupled map systems. The concept was to consider the widest class of systems possible, but for the sake of clarity and brevity we refrained from some superficial generalizations which pose no mathematical complications. For example, $g|_{(-1/2,1/2)}$ need not be the identity, a very similar version of our results is most likely to hold when $g|_{(-1/2,1/2)}$ has sufficient smoothness properties and bounds on the derivative. On the other hand, our result is somewhat less explicit then the stability result obtained in \cite{selley2016mean}, since we do not have an expression for $f_*^{\varepsilon}$, we only managed to show that it is of order $\varepsilon$ distance from $f_*^0$. A special case when it can be made explicit is when the unique acim of $T$ is Lebesgue. Actually in this case the proof in \cite{selley2016mean} can be applied with minor modifications showing that the constant density is a stable invariant distribution of the coupled map system for sufficiently small coupling strength.

In the introduction we remarked that if we choose the initial measure to be an average of $N$ point masses, we get a coupled map system of finite size. This can more conveniently be represented by a dynamical system on $\mathbb{T}^N=\mathbb{T}\times \dots \times \mathbb{T}$ with piecewise $C^2$ dynamics (the specific smoothness depending on the smoothness of $T$). Our papers \cite{selley2016mean} and \cite{selley2016} suggest that the analysis in the case of finite system size is particularly complicated and necessitates a thorough geometrical understanding. A direct consequence of this is that little can be proved when the dimension is large. If for example $g$ were a smooth function on $\mathbb{T}$, it could be shown that there exists a unique absolutely continuous invariant measure $\mu_{N}^{\varepsilon}$ for any system size $N$, once $\varepsilon$ is smaller then some $\varepsilon^*$ which does not depend on $N$. Simulations suggest that this is likely to be the case also when $g$ is defined by \eqref{gg}. If this could be verified, one could even aim to prove an analogue of part 1) of Theorem 3 in \cite{bardet2009stochastically}. Namely, along the same lines as in \cite{bardet2009stochastically}, one could show that for sufficiently small $\varepsilon$, the sequence $\mu_{N}^{\varepsilon} \circ \epsilon_N^{-1}$ converges weakly to $\delta_{\mu_*^{\varepsilon}}$, where
\[
\epsilon_N(x_1,\dots,x_N)=\frac{1}{N}\sum_{i=1}^N \delta_{x_i},
\]
and $\mu_*^{\varepsilon}$ is the unique invariant measure for the infinite system with coupling strength $\varepsilon$.

Finally, in our opinion the truly interesting question is the spectrum of possible limit behaviors in our class of systems. Are stability and synchronization, as shown in this paper to exist, the only possibilities? For example, one can easily imagine that for stronger coupling than the one producing stability, multiple locally attracting or repelling fixed densities can arise -- as seen in the case of the model introduced in \cite{bardet2009stochastically}. However, one would have to develop completely new analytical tools to prove such statements in our class of models. So on this end there is plenty of room for innovation.

\section*{Acknowledgements}
A large part of this work was done during the stay of FMS at the Department of Mathematics of the University of Erlangen-Nuremberg. The hospitality of the Department, along with the support of the DAAD is thankfully acknowledged. The research of PB, FMS and IPT was supported in part by Hungarian National Foundation for Scientific Research (NKFIH OTKA) grants K104745 and K123782.

\newpage

\bibliographystyle{plain}

\end{document}